\documentclass[reqno,11pt]{amsart}
\usepackage{amsfonts}
\usepackage{times}
\usepackage{amssymb, amsmath}
\allowdisplaybreaks

\date{January 18, 2012}
\theoremstyle{plain}
\newtheorem{theorem}{Theorem}[section]

\newtheorem{lemma}[theorem]{Lemma}

\theoremstyle{definition}
\newtheorem{defn}{Definition}[section]
\theoremstyle{remark}
\newtheorem{rem}{Remark}[section]
\numberwithin{equation}{section}

\newcommand{\real}{\mathbb{R}}
\newcommand{\ds}{\displaystyle}
\newcommand{\vare}{\varepsilon}
\newcommand{\loc}{\scriptsize{loc}}

\begin{document}

\title[Stability of 2D flows under three-dimensional perturbations]
{Stability of two-dimensional viscous incompressible flows under three-dimensional perturbations and inviscid symmetry breaking}

\author[C. Bardos]{Claude Bardos}
\address[C. Bardos] {31 Avenue Trudaine 75009 Paris\\
Laboratoire J.-L. Lions\\
 Universite de Paris VII, Denis Diderot}
\email{claude.bardos@gmail.com}

\author[M.C. Lopes Filho]{Milton C. Lopes Filho}
\address[M. C. Lopes Filho]
{Depto. de Matem\'{a}tica, IMECC\\
Rua S\'{e}rgio Buarque de Holanda, 651\\
Universidade Estadual de Campinas - UNICAMP\\
 Campinas, SP 13083-859, Brasil}
\email{mlopes@ime.unicamp.br}

\author[D. Niu]{Dongjuan Niu}
\address[D. Niu]
{School of Mathematical Sciences\\
Capital Normal University \\
Beijing 100048, P. R. China} \email{djniu@cnu.edu.cn}

\author[H.J. Nussenzveig Lopes]{Helena J. Nussenzveig Lopes}
\address[H.J. Nussenzveig Lopes]
{Depto. de Matem\'{a}tica, IMECC\\
Rua S\'{e}rgio Buarque de Holanda, 651\\
Universidade Estadual de Campinas - UNICAMP\\
 Campinas, SP 13083-859, Brasil}
\email{hlopes@ime.unicamp.br}

\author[E.S. Titi]{Edriss S. Titi}
\address[E.S. Titi]
{Department of Mathematics \\
and Department of Mechanical and  Aerospace Engineering \\
University of California \\
Irvine, CA  92697-3875, USA. {\bf Also:}
 Department of Computer Science and Applied Mathematics \\
Weizmann Institute of Science  \\
Rehovot 76100, Israel. Fellow of the Center of Smart Interfaces (CSI), Technische Universit\"{a}t Darmstadt, Germany. } \email{etiti@math.uci.edu}
\email{edriss.titi@weizmann.ac.il}

\begin{abstract}
In this article we consider weak solutions of the three-dimensional incompressible fluid flow equations with initial data admitting
a one-dimensional symmetry group.  We examine both the viscous and inviscid cases. For the case of viscous flows, we
prove that Leray-Hopf weak solutions of the three-dimensional Navier-Stokes equations preserve initially imposed symmetry and that such
symmetric flows are stable under general three-dimensional perturbations, globally in time. We work in three different contexts: two-and-a-half-dimensional, helical and
axi-symmetric flows. In the inviscid case, we observe that, as a consequence of recent work by De Lellis and Sz\'ekelyhidi, there are genuinely
three-dimensional weak solutions of the Euler equations with two-dimensional initial data. We also present two partial results where restrictions on
the set of initial data, and on the set of admissible solutions rule out spontaneous symmetry breaking; one is due to P.-L. Lions and the other is a consequence of our viscous stability result.
\end{abstract}

\maketitle




\section{Introduction}

 In this article, we consider the equations for incompressible fluid motion:

\begin{equation} \label{NS}
\left\{ \begin{array}{l}
\partial_t u + (u \cdot \nabla) u = -\nabla p + \nu \Delta u + f, \\
\mbox{div } u = 0,
\end{array} \right. \end{equation}
supplemented by appropriate initial and boundary data. Above, $u = (u_1,u_2,u_3)$ is the fluid velocity and $p$ is the scalar pressure. The
external force $f$ and the kinematic viscosity $\nu \geq 0$ are given. System \eqref{NS} is referred to as the Navier-Stokes equations
in the viscous case $(\nu > 0)$, and as the Euler equations of ideal fluid motion in the inviscid case $(\nu = 0)$.

Let $u=u(t,x)=u(t,x_1,x_2,x_3)$ be a Leray-Hopf weak solution (see Definition \ref{LH}) of the Navier-Stokes equations \eqref{NS} for some $\nu > 0$ in a domain $\Omega$, with zero forcing. Assume that the domain $\Omega \subset \real^3$ and the initial velocity $u_0 \equiv u(0,\cdot)$  are symmetric
with respect to a one-parameter group which is invariant under the Navier-Stokes evolution. For example, one may think of flow in the full
three-dimensional space, which is periodic in all three directions, for which the initial velocity is periodic and invariant
under vertical translations, i.e., whose components do not depend on the vertical variable.  Our main new result is global-in-time stability in
the energy space of solutions which preserve the symmetry, within the class of Leray-Hopf weak solutions of the three-dimensional Navier-Stokes equations.
As a consequence, any Leray-Hopf weak solution of the three dimensional Navier-Stokes equations which starts symmetric will stay symmetric for positive time, ruling out spontaneous symmetry breaking within this class of weak solutions.  We will also see that, as a special case of a construction due to C. De Lellis and L. Sz\'{e}kelyhidi, spontaneous symmetry breaking does occur among weak solutions of the three-dimensional Euler equations.

Our analysis, in the viscous case, is closely related to weak-strong uniqueness results for the Leray-Hopf weak solutions of the incompressible 3D Navier-Stokes
system, a subject with an old, large and deep literature. Symmetric flows, regarded as a special class of three-dimensional flows, are more
regular than a general three-dimensional Leray-Hopf weak solution. The idea behind weak-strong uniqueness is to impose additional regularity assumptions on a
given weak solution in order to guarantee it is unique. Our point of departure is whether this additional regularity of symmetric weak solutions is enough
to ensure uniqueness.

  The first weak-strong uniqueness result for Leray-Hopf solutions of the Navier-Stokes equations is due to Sather and Serrin, see \cite{Serrin62}, and it is usually referred to as Sather-Serrin Uniqueness Criterion, see also the work of G. Prodi \cite{prodi}. Briefly stated, a weak solution in $L^q((0,T),L^p(\Omega))$ is unique if $ 3/p + 2/q  = 1$, $3<p<\infty$. Recently this criteria was extended to the  limit case $p=3,\, q=\infty$, see \cite{SeSv,ESeSv1,ESeSv2}.

  General two-dimensional flows, for example, are in $L^{\infty}((0,T);L^2) \cap L^2((0,T);H^1)$, which, by interpolation and Sobolev imbedding, are in $L^q((0,T);L^p(\Omega))$ with $2 \leq p < \infty$, $2 \leq q < 2p/(p-2)$ or $(p,q) = (2,\infty)$. We call this region of the extended $(p,q)$-plane $\mathcal{R}$. The hyperbola $3/p + 2/q = 1$ lies strictly above the region $\mathcal{R}$, approaching $\mathcal{R}$ only as $(p,q) \to (\infty,2)$. Hence, the Sather-Serrin criterion (or its extension to the limit case $p=3,\, q=\infty$) does not ensure uniqueness of two-dimensional flows, when viewed as three-dimensional flows. The terminology {\it two-dimensional flows}, in this work, means that the components of the velocity fields   do not depend on the vertical variable, $x_3$. We observe that, depending on the context,  the velocity fields of two-dimensional flows can have either two or three components. We also recall that two-dimensional flows are sometimes  called two-and-a-half-dimensional flows (denoted $2\frac{1}{2}D$  flows) when the velocity field has three non-trivial components (see, e.g., Section 2.3.1 of \cite{MB2002}) .

There is a large literature dedicated to extensions of the Sather-Serrin criterion, see \cite{Germain2006} and references therein. However, the results which have been obtained tend to obey the same scaling as the Sather-Serrin condition. The problems treated in the current paper are, in a sense, off-scale, and, therefore, only the extensions which have been obtained near the critical case $(\infty,2)$ are potentially relevant to our work. One particularly noteworthy result was established by H. Kozono and Y. Taniuchi, see \cite{KT2000} and it concerns extending the Sather-Serrin uniqueness criterion to vector fields which are bounded in $L^{\infty}((0,T);L^2(\real^3)) \cap L^2((0,T);BMO(\real^3))$. In fact, vector fields which are in $L^2((0,T;\dot{H^1}(\real^2))$, such as the solutions of the two-dimensional Navier-Stokes equations, are actually bounded in $L^2((0,T);BMO(\real^3))$ because $\dot{H^1}(\real^2) \subset BMO(\real^2) \subset BMO(\real^3)$ (we note emphatically that these vector fields are independent of the third variable). However, these vector fields are not square-integrable in $\real^3$, so we cannot use Kozono and Taniuchi's criterion to address uniqueness (or stability) of two-dimensional solutions viewed as three-dimensional flows.

The original argument in \cite{Serrin62} was formulated in an arbitrary domain, but, as in  Kozono and Taniuchi's result, the extensions have been full-space results making use of harmonic analysis machinery. Of course, to circumvent the fact that two-dimensional flows are not square integrable in full space, one should look for uniqueness among 3D flows in another domain, such as flows which are periodic in the third variable. It seems likely that one could adapt the proof of Kozono and Taniuchi's criterion to flows which are periodic in the third variable, and then obtain an uniqueness and stability result along the lines suggested above. However, in this work, we would like to take a more elementary approach, closer to Sather and Serrin's original argument, which works on a vertically periodic flow in cylindrical domain of general shape, and in other situations as well.

The problem of stability of two-dimensional flows under three-dimensional perturbations is very natural and interesting from the physical point of view, and it has been the subject of previous work. The first results in this direction were obtained by G. Ponce, R. Racke, T.C. Sideris and E.S. Titi, see \cite{PRST94}. Their main result is global existence of a strong solution which starts close, in $H^1$ to a two-dimensional solution, also a stability estimate. Their result was later improved in \cite{Iftimie99,Mucha08}, by relaxing regularity conditions on the perturbation, but always working in the class of strong solutions, and therefore, focusing their concern on global existence, rather than stability. Our work may be regarded as an extension of these articles to weak Leray-Hopf solutions.

We are going to prove uniqueness and stability results for Leray-Hopf weak solutions in three different contexts:
\begin{enumerate}
\item[(i)] two-dimensional flow in an infinite straight cylinder with bounded and smooth cross-section, with no-slip boundary condition and three-dimensional perturbations which are periodic in the vertical direction;
\item[(ii)] helical flow in a straight circular cylinder, with no-slip boundary condition and general three-dimensional perturbations with the same period as the helical flow;
\item[(iii)] axi-symmetric flow in the interior of an axi-symmetric torus with smooth cross-section bounded away from the symmetry axis, no-slip boundary condition and a general three dimensional perturbation.
\end{enumerate}

In each case, existence of a symmetric weak solution for the Navier-Stokes equations when the initial data is symmetric is an implicit requirement of our analysis, and can be obtained by an easy adaptation of the classical argument by Leray.  Global well-posedness of weak solutions is also known in all three cases. We emphasize that these well-posedness results all refer to the corresponding symmetry-reduced equations.
For case (i), it was pointed out in Proposition 2.7 of \cite{MB2002}, that global existence of two-dimensional flows, regarded as three-dimensional flows, reduces to the global well-posedness result of weak solutions of the standard 2D Navier-Stokes equations in a bounded domain, which was established in \cite{Ladyzhenskaya1959}. For global existence and uniqueness of weak solutions in case (ii) see \cite{MTL} and, in case   (iii), see \cite{LadyBook,UY}.

Concerning the inviscid case, we discuss three results. The first result is existence of a genuinely 3D weak solution of the 3D Euler equations in a periodic cube, with two-dimensional initial data. The result is a special case of a construction by De Lellis and Sz\'ekelyhidi in \cite{DS10}, as formulated by Wiedemann in \cite{Wiedemann}. The second result is a consequence of a weak-strong uniqueness theorem
for dissipative solutions of the Euler equations, due to P.-L. Lions in \cite{pllions}. The third result is a corollary of our viscous stability estimates, applied to Euler solutions which are vanishing viscosity limits in a specific way.

The remainder of this work is divided into four sections. Section 2 contains basic definitions and notation, Section 3 concerns the viscous results, Section 4 contains the inviscid results, and Section 5 is final comments and conclusions.

\section{Preliminaries}

In this section we fix notation and set down some basic definitions. In this article, we are concerned with incompressible flows in three contexts - triply periodic flow in a box, flow in an infinite vertical
cylinder whose horizontal cross sections are bounded and smooth and which are periodic in the vertical direction and flows in a bounded axi-symmetric domain. To discuss the incompressible
Navier-Stokes equations in these contexts we first introduce the Hilbert spaces usually denoted by
$H$ and $V$ in the literature, adjusting things according to the specific case as follows:

\begin{enumerate}

\item for the periodic box $\Omega \equiv (0,1)^3$, the spaces $H(\Omega)$ and $V(\Omega)$ are the closure of the $C^{\infty}$, periodic, divergence-free
vector fields in $\Omega$ with respect to the $L^2$ and $H^1$ norms in $\Omega$, respectively.

\item for the periodic cylinder $\Omega = D \times (0,L)$, where $L>0$ and $D$ is a bounded smooth domain in $\real^2$, the spaces $H(\Omega)$ and $V(\Omega)$ are the closure  of the $C^{\infty}$ vector fields which are, periodic in the vertical variable, compactly supported in the horizontal sections and divergence-free in $\Omega$ with respect to the $L^2$ and $H^1$ norms in $\Omega$, respectively.

\item for a smooth axi-symmetric domain $\Omega$,   the spaces $H(\Omega)$ and $V(\Omega)$ are the closure  of the $C^{\infty}$ vector fields which are, compactly supported and divergence-free in $\Omega$ with respect to the $L^2$ and $H^1$ norms in $\Omega$, respectively.
\end{enumerate}

We denote by $\mathcal{D}(\Omega)$ the space of $C^{\infty}$ test functions,  periodic in the case of the cube, compactly supported for the axi-symmetric domain and and periodic in the vertical variable, compactly supported in the horizontal direction for the periodic cylinder. We will
also use the notation $H_w(\Omega)$ for the vector space $H(\Omega)$ endowed with the weak topology in $L^2$.

Let us recall the definition of a Leray-Hopf weak solution of the Navier-Stokes system:

\begin{defn} \label{LH}
Let $\Omega$ be either the periodic box, the periodic cylinder or an axi-symmetric domain as above and let
$u \in C^0([0,\infty);H_w(\Omega)) \cap L^{\infty}((0,\infty);H(\Omega)) \cap L^2_{\loc}([0,\infty);V(\Omega))$.
Then $u$ is a weak Leray-Hopf solution of \eqref{NS} with initial data $u_0 \in H(\Omega)$ and forcing $f \in L^2((0,T);H^{-1}(\Omega))$ if:
\begin{enumerate}

\item for any test function $\phi \in C^{\infty}_c([0,\infty);\mathcal{D}(\Omega))$ such that $\mbox{ div} \phi = 0$ we have:
\[\int_0^{\infty} \int_{\Omega} \left\{ - \partial_t \phi \cdot u - [( u \cdot \nabla) \phi ]\cdot u + \nu \nabla u : \nabla \phi \right\} \,dx dt - \int_{\Omega} u_0 \cdot \phi(0,x)\, dx \]
\[= \int_0^{\infty}\langle \phi(t,\cdot), f(t,\cdot) \rangle \,dt,\]
where $A:B \equiv \sum_{i,j} a_{ij}b_{ij}$ is the trace product of two matrices and $\langle \cdot,\cdot \rangle$ denotes the duality pairing between $H^1_0(\Omega)$ and $H^{-1}(\Omega)$.

Additionally,

\item for any $t > 0$,

\begin{equation} \label{eneest}
\|u(t,\cdot)\|_{L^2(\Omega)}^2 + 2 \nu \int_0^t \|\nabla u(s,\cdot)\|_{L^2(\Omega)}^2 ds \leq \| u_0\|_{L^2(\Omega)}^2 + 2 \int_0^t \langle u(s,\cdot),f(s,\cdot)\rangle \, ds.
\end{equation}

\end{enumerate}
\end{defn}

Note that by density arguments, and the continuity of the terms used in the identity in item (1) of Definition \ref{LH}, one can extend the  Definition \ref{LH}  to allow for the use of test functions $\phi \in C^{\infty}_c([0,\infty);V(\Omega))$  in the identity in item (1) of Definition \ref{LH}.

We also require a definition of weak solution for the Euler equations, but only in the case of the periodic box and without forcing, see \cite{DS10,Wiedemann}.

\begin{defn}
Let $\Omega = (0,1)^3$ be the periodic box and let $u \in C^0([0,\infty);H_w(\Omega))$. We say that $u$ is a weak
solution of the Euler equations (\eqref{NS}, $\nu = 0$) with initial velocity $u_0 \in H(\Omega)$
if  for any test function $\phi \in C^{\infty}_c([0,\infty);\mathcal{D}(\Omega))$ such that $\mbox{ div} \phi = 0$ we have:
\[ \int_0^{\infty} \int_{\Omega} \left \{ \partial_t \phi \cdot u + [( u \cdot \nabla) \phi ]\cdot u  \right\}\,dx dt + \int_{\Omega} u_0 \cdot \phi(0,x) \, dx = 0.\]
\end{defn}

\section{Viscous flow}

In this section we will state and prove stability results for Leray-Hopf weak solutions of the Navier-Stokes equations \eqref{NS}, with zero forcing, in the three contexts described in the introduction.

We start with three-dimensional perturbations of two-dimensional flows. Recall that two-dimensional flows refer to solutions of the three-dimensional Navier-Stokes equations which are independent of $x_3$.

\begin{theorem}\label{2andHalfD}
Let $D \subset \real^2$ be a bounded domain with smooth boundary. Consider $u_0 \in H(D)$ and let $u \in C^0([0,\infty);H_w(D)) \cap L^2((0,\infty);V(D))$ be the unique weak solution of the $2\frac{1}{2}D$  incompressible Navier-Stokes equations having, as initial data, $u_0$. Fix $L>0$ and set $C = D \times (0,L)$. Let $v \in L^{\infty}((0,\infty);H(C)) \cap L^2((0,\infty);V(C))$ be a Leray-Hopf weak solution of the three-dimensional incompressible Navier-Stokes equations with initial data $v_0$, where $v_0 \in H(C)$.
The following estimate holds true:
\[\|v-u\|_{L^2(C)}^2 (t)\leq \|v_0-u_0\|_{L^2(C)}^2 \exp{\left(\frac{27}{64\nu^4}\|u_0\|_{L^2(D)}^4\right)}, \mbox{ for all } t \geq 0.\]
\end{theorem}
\begin{proof}[Proof of Theorem \ref{2andHalfD}]
We begin by following the argument in the proof of Theorem 6 in \cite{Serrin62}. Fix $T>0$ and let $\eta_{\vare}$ be a standard $1$-dimensional mollifier (smooth, non-negative,  even, supported in $(-\vare,\vare)$ and with unit integral). Define
\[u^{\vare} = u^{\vare}(t,x) \equiv \int_0^T \eta_{\vare}(t-s)u(s,x)\,ds.\]
Define in an analogous manner $v^{\vare} = v^{\vare}(t,x)$.

Then, using $u^{\vare}$ as test function in the weak formulation of the equation for $v$, and $v^{\vare}$ as test function for the equation for $u$, we find the following two identities:
\begin{eqnarray}
-(u,v^{\vare})(T) + \int_0^T (u,\partial_t v^{\vare})\,ds - \nu \int_0^T (\nabla u,\nabla v^{\vare})\,ds \\ \nonumber = - \int_0^T ((u\cdot\nabla)v^{\vare},u)\,ds - (u_0,v^{\vare}_0);
\end{eqnarray}

\begin{eqnarray}
-(v,u^{\vare})(T) + \int_0^T (v,\partial_t u^{\vare})\,ds - \nu \int_0^T (\nabla v,\nabla u^{\vare})\,ds \\ \nonumber = - \int_0^T ((v\cdot\nabla)u^{\vare},v)\,ds - (v_0,u^{\vare}_0),
\end{eqnarray}
where $(\cdot,\cdot)$ denotes the inner product in $L^2(C)$.

We add these two identities, using the fact that $\int_0^T (u,\partial_t v^{\vare})\,ds = - \int_0^T (v,\partial_t u^{\vare})\,ds$, and we find

\begin{eqnarray} \label{deducefromweakform}
-(u,v^{\vare})(T) -(v,u^{\vare})(T) - \nu \int_0^T [(\nabla u,\nabla v^{\vare}) + (\nabla v,\nabla u^{\vare})]\,ds \\ \nonumber = - \int_0^T [((u\cdot\nabla)v^{\vare},u) + ((v\cdot\nabla)u^{\vare},v) ]\,ds - (u_0,v^{\vare}_0) - (v_0,u^{\vare}_0).
\end{eqnarray}

We multiply \eqref{deducefromweakform} by $2$ and let $\vare \to 0$ to obtain

\begin{eqnarray} \label{deducefromweakformGOOD}
-2(u,v)(T) - 4\nu \int_0^T (\nabla u,\nabla v) \,ds \\ \nonumber = 2\int_0^T (((v-u)\cdot\nabla)(v-u),u)  \,ds - 2(u_0,v_0).
\end{eqnarray}

Next, we use the energy inequality, satisfied by both $u$ and $v$ (see Definition \ref{LH}):

\begin{equation} \label{eninequ1}
\|u\|_{L^2(C)}^2 (T) + 2\nu \int_0^T \|\nabla u (s,\cdot)\|_{L^2(C)}^2 \,ds \leq \|u_0\|_{L^2(C)}^2;
\end{equation}

\begin{equation} \label{eninequ2}
\|v\|_{L^2(C)}^2 (T) + 2\nu \int_0^T \|\nabla v(s,\cdot)\|_{L^2(C)}^2 \,ds \leq \|v_0\|_{L^2(C)}^2.
\end{equation}

Introduce $w \equiv v - u$ and add \eqref{deducefromweakformGOOD}, \eqref{eninequ1} and \eqref{eninequ2} to find

\begin{equation} \label{diffineqw}
\|w\|_{L^2(C)}^2 (T) + 2\nu \int_0^T \|\nabla w\|_{L^2(C)}^2 \,ds \leq \|w_0\|_{L^2(C)}^2 + 2\int_0^T (w\cdot\nabla w, u) \,ds.
\end{equation}

This is precisely inequality (27) in \cite{Serrin62}. At this point we depart from the argument presented in \cite{Serrin62} and use the fact that $u$ is two-dimensional. We analyze the nonlinear term using the two-dimensional Ladyzhenskaya inequality in $D$:

\[
\int_0^T ((w\cdot\nabla)w,u)\,ds = \int_0^T\int_0^L\int_D [(w\cdot\nabla)w] \cdot u \,dx_1dx_2dx_3ds \]
\[\leq \int_0^T\int_0^L \|w\|_{L^4(D)}\|\nabla w\|_{L^2(D)}\|u\|_{L^4(D)}\,dx_3ds
\]
\[\leq 2^{1/4} \int_0^T\int_0^L \|w\|_{L^2(D)}^{1/2}\|(\partial_{x_1},\partial_{x_2}) w\|_{L^2(D)}^{1/2}\|\nabla w\|_{L^2(D)}\|u\|_{L^4(D)}\,dx_3ds
\]
\[\leq 2^{1/4} \int_0^T\int_0^L \|w\|_{L^2(D)}^{1/2}\|\nabla w\|_{L^2(D)}^{3/2}\|u\|_{L^4(D)}\,dx_3ds
\]
\[\leq \nu\int_0^T\int_0^L \|\nabla w\|_{L^2(D)}^2 \,dx_3ds  + \frac{27}{128\nu^3}\int_0^T\int_0^L \|w\|_{L^2(D)}^2  \|u\|_{L^4(D)}^4\,dx_3ds,
\]
by Young's inequality. Therefore, using the fact that $\|u(s,\cdot)\|_{L^4(D)}^4$ is independent of $x_3$, we obtain

\begin{equation} \label{chato1}
\int_0^T ((w\cdot\nabla)w,u)\,ds
\leq \nu\int_0^T\|\nabla w\|_{L^2(C)}^2 \,ds  + \frac{27}{128\nu^3}\int_0^T \|u\|_{L^4(D)}^4 \|w\|_{L^2(C)}^2  \,ds.
\end{equation}

We input \eqref{chato1} in \eqref{diffineqw} to find

\begin{equation} \label{diffineqw2}
\|w\|_{L^2(C)}^2 (T)  \leq \|w_0\|_{L^2(C)}^2 +\frac{27}{64\nu^3} \int_0^T \|u\|_{L^4(D)}^4 \|w\|_{L^2(C)}^2 \,ds.
\end{equation}

Therefore, by Gronwall's Lemma we deduce that
\begin{equation} \label{chato2} \|w\|_{L^2(C)}^2(T) \leq \|w_0\|_{L^2(C)}^2 \exp{\left(\frac{27}{64\nu^3}\int_0^T \|u\|_{L^4(D)}^4\,ds\right)}.\end{equation}

Finally,  we use again the Ladyzhenskaya inequality to estimate:

\[ \int_0^T \|u\|_{L^4(D)}^4 ds \leq \ 2 \int_0^T \| u \|_{L^2(D)}^2 \|\nabla u\|_{L^2(D)} ^2 ds  \]
\[ \leq 2  \| u \|_{L^{\infty}((0,T);L^2(D))}^2  \| \nabla u \|_{L^2((0,T);L^2(D))}^2,\]
which, using \eqref{eninequ1} together with the fact that $u$ is independent of $x_3$ yields the desired result, once we replace $T$ by an arbitrary $t\geq 0$
and notice that the dependence on $L$ cancels out.
\end{proof}

\begin{rem} An immediate corollary of Theorem \ref{2andHalfD} is the uniqueness of Leray-Hopf weak solutions for two-dimensional initial data.
\end{rem}

\vspace{0.5cm}

Next, we will examine a variant of Theorem \ref{2andHalfD},
pertaining to helical flows.

A vector field $U$ is called helical, with {\it step} $\sigma \in \real \setminus \{0\}$ if, for any $\theta \in \real$ and any $x \in \real^3$,
\[U\left(\left[\begin{array}{ccc}
\cos\theta & \sin\theta & 0 \\
-\sin\theta & \cos\theta & 0 \\
0 & 0 & 1 \end{array} \right]x + \left[\begin{array}{c}
															0 \\ 0 \\ \frac{\sigma}{2\pi}\theta \end{array} \right]\right) = U(x).\]
We refer the reader to \cite{MTL} for well-posedness results for the Navier-Stokes equations with helical symmetry. For simplicity, we will focus on the special case of helical flows in a straight circular pipe.

\begin{theorem} \label{Helical}
Let $D$ be the unit disk in the plane, while $C$ denotes the unit cylinder
$D \times (0,1)$. Let $u_0 \in H(C)$ be a helical vector field
with step equal to $1$. Let $u \in C^0([0,\infty);H_w(C)) \cap
L^2((0,\infty);V(C))$ be the unique weak solution of the helical
incompressible Navier-Stokes equations having, as initial data,
$u_0$, given in Theorem 3.3 of \cite{MTL}. Let $v_0 \in H(C)$ and let $v \in
C^0([0,\infty);H_w(C)) \cap L^2((0,\infty);V(C))$ be a
Leray-Hopf weak solution of the three-dimensional incompressible
Navier-Stokes equations with initial data $v_0$.
Then, the following inequality is valid:
\[\|v-u\|_{L^2(C)}^2 (t)\leq \|v_0-u_0\|_{L^2(C)}^2 \exp{\left(\frac{27}{64\nu^4}\|u_0\|_{L^2(C)}^4\right)}, \mbox{ for all } t \geq 0.
\]
\end{theorem}

\begin{proof}[Proof of Theorem \ref{Helical}]
We can use the same proof as for the $2\frac{1}{2}D$ case once we make the following observations:
\begin{verse}
(i) the $L^p(D)$-norms of $u$ are independent of $x_3$, for any $p\geq 1$; \\
(ii) the $L^2(D)$-norm of $(\partial_{x_1},\partial_{x_2}) u$ is independent of $x_3$ and bounded above by the $L^2(C)$-norm of $\nabla u$.
\end{verse}
\end{proof}

\begin{rem} As before, this easily yields uniqueness of Leray-Hopf weak solutions with helical initial data.
\end{rem}

\vspace{0.5cm}

Lastly, we discuss the case of axi-symmetric flows.

\begin{theorem} \label{Axisymmetric}
Let $D$ be a bounded, smooth domain compactly contained in $\{(r,z)\;|\; 0<r<\infty,\,z \in \real\}$ and set $C = \{(r,z,\theta)\;|\;(r,z)\in D,\,0\leq\theta \leq 2\pi\}$. Let $u_0 \in H(C)$ be an axially symmetric vector field. Let $u \in C^0([0,\infty);H_w(C)) \cap L^2((0,\infty);V(C))$ be the unique weak solution of the  axi-symmetric incompressible Navier-Stokes equations having, as initial data, $u_0$, given in \cite{LadyBook,UY}. Let $v_0 \in H(C)$ and let $v \in C^0([0,\infty);H(C)) \cap L^2((0,\infty);V(C))$ be a Leray-Hopf weak solution of the three-dimensional incompressible Navier-Stokes equations with initial data $v_0$. There exists a constant $M=M(D,\nu) >0$ such that the following inequality is valid:
\[\|v-u\|_{L^2(C)}^2 (t)\leq \|v_0-u_0\|_{L^2(C)}^2 \exp{\left(M\|u_0\|_{L^2(C)}^4\right)}, \mbox{ for all } t \geq 0.
\]
\end{theorem}

\begin{proof}[Proof of Theorem \ref{Axisymmetric}]
We must make small modifications of the proof for the $2\frac{1}{2}D$ case, beginning by writing the integral over $C$ as $\int_0^{2\pi}\int_{D}$ with respect to the measure $rdrdzd\theta$.

We estimate the nonlinear term as follows:
\[
\int_0^T ((w\cdot\nabla)w,u)\,ds = \int_0^T\int_0^{2\pi}\int_D [(w\cdot\nabla)w] \cdot u \,rdrdzd\theta ds \]
\[\leq \int_0^T\int_0^{2\pi} \|w\|_{L^4(D,rdrdz)}\|\nabla w\|_{L^2(D,rdrdz)}\|u\|_{L^4(D,rdrdz)}\,d\theta ds
\]
\[\leq K\int_0^T\int_0^{2\pi} \|w\|_{L^2(D,rdrdz)}^{1/2}\|(\partial_{r},\partial_z) w\|_{L^2(D,rdrdz)}^{1/2}\|\nabla w\|_{L^2(D,rdrdz)}\|u\|_{L^4(D,rdrdz)}\,d\theta ds,
\]
where $K>0$ is a constant appearing in the two-dimensional Ladyzhenskaya inequality in $D$, valid since $w$ vanishes on the boundary of $D$ for each
fixed $\theta$, together with the fact that $D$ is bounded away from the axis of symmetry, so that $r>a$, for some fixed $a>0$;
\[\leq K\int_0^T\int_0^{2\pi} \|w\|_{L^2(D,rdrdz)}^{1/2}\|\nabla w\|_{L^2(D,rdrdz)}^{3/2}\|u\|_{L^4(D,rdrdz)}\,d\theta ds
\]
\[\leq \nu \int_0^T\int_0^{2\pi} \|\nabla w\|_{L^2(D,rdrdz)}^2 \,d\theta ds  + \widetilde{K} \int_0^T\int_0^{2\pi} \|w\|_{L^2(D,rdrdz)}^2  \|u\|_{L^4(D,rdrdz)}^4\,d\theta ds,
\]
for some $\widetilde{K} >0$, resulting from using Young's inequality,
\[\leq \int_0^T\|\nabla w\|_{L^2(C)}^2 \,ds  + \widetilde{K} \int_0^T \frac{1}{2\pi}\|u\|_{L^4(C)}^4 \|w\|_{L^2(C)}^2  \,ds,
\]
where we have used the fact that $\|u(\cdot,\theta)\|_{L^4(D,rdrdz)}^4$ is independent of $\theta$ and is equal to $(1/2\pi)\|u\|_{L^4(C)}^4$. We observe that, above, the constant $K$ depends on $a$.

By the Gronwall Lemma we deduce, as before, that
\[\|w\|_{L^2(C)}^2(T) \leq \|w_0\|_{L^2(C)}^2 \exp{\left(\frac{\widetilde{K}}{2\pi}\int_0^T \|u\|_{L^4(C)}^4\,ds\right)}.\]

Finally, we use again the two-dimensional Ladyzhenskaya inequality for $u$, noticing that the derivatives which appear are with respect to $r$ and $z$ and, hence, their $L^2(D,rdrdz)$-norms are independent of $\theta$. This, together with the energy inequality \eqref{eninequ1}, yields the desired result, replacing $T$ by an arbitrary time $t \geq 0$.
This concludes the proof.

\end{proof}

Global existence and uniqueness of weak solutions for the axi-symmetric Navier-Stokes equations was established by O. Ladyzhenskaya, see \cite{Ladyzhenskaya1959}, but only under the assumption that the axi-symmetric fluid domain be bounded away from the symmetry axis, i.e., $r>a$, for some $a>0$. This restriction has the same origin as in Theorem \ref{Axisymmetric}, namely, loss of essential 2D scaling at the symmetry axis. (Additional results on global regularity of special solutions of the axi-symmetric Navier-Stokes equations, defined in a domain which includes the symmetry axis, have been obtained in \cite{houli}.)  We note that Theorem \ref{Axisymmetric}   leaves open the possibility that there might exist Leray-Hopf weak solutions of the (3D) Navier-Stokes equations with $L^2$ axi-symmetric initial velocity, for which the symmetry is spontaneously broken.

\begin{rem}
We have considered, throughout this section, viscous flows with zero forcing. It should be noted that, if the forcing term $f$ does not vanish and respects the same symmetry as the initial velocity, then the proofs of Theorems \ref{2andHalfD}, \ref{Helical} and \ref{Axisymmetric} can be easily adapted to show that
\[\|v-u\|_{L^2(C)}^2 (t)\leq \|v_0-u_0\|_{L^2(C)}^2 \exp\left\{M\left(\|u_0\|_{L^2(C)}^2 + 2\int_0^t \langle u (s,\cdot) , f(s,\cdot) \rangle \, ds \right)^2\right\},\]
for some $M=M(D,\nu)>0$. This implies, clearly, continuous dependence with respect to initial data and, in particular, uniqueness.
\end{rem}

\section{Inviscid flow}

In this section we discuss the possibility of spontaneous symmetry breaking for the Euler system.
Our first observation is that spontaneous symmetry breaking is possible for weak solutions of the
Euler system, in contrast with what we observed for the Navier-Stokes equations. This is a special case of a construction due to De Lellis and Sz\'ekelyhidi in \cite{DS10},
see Proposition 2. We will use this construction as formulated in Theorem 2 of \cite{Wiedemann}.  Before we begin we
need to introduce some terminology. Since, in this section, we deal only with flows in a periodic box we introduce the notation $Q^N = [0,1]^N$
for the periodic box in $\real^N$.

\begin{defn} \label{planar}
Let $f \in L^1(Q^3)$. We say that $f$ is essentially independent of $x_3$ (which is shortened to ei-$x_3$) if, for almost every $a,b \in (0,1)$, $f(x_1,x_2,a) = f(x_1,x_2,b)$, for
almost all $(x_1,x_2) \in Q^2$.
\end{defn}

With this, we are now ready to state precisely the symmetry breaking result.

\begin{theorem} \label{DSz} Let $u_0 = (u_0^1,u_0^2)\in C^{\infty}(Q^2)$ be divergence-free and periodic.
There exists a  weak solution (in fact infinitely many)  $u=u(t,x_1,x_2,x_3) \in C^0([0,\infty);H_w(Q^3))$ of the incompressible $3D$ Euler
equations such that $u(t=0) = (u_0,0)$, and $u$ is not ei-$x_3$.
\end{theorem}

\begin{proof}
Let $v = v(t,x_1,x_2)$ be the unique solution of the 2D Euler system in $Q^2$, given in \cite{EM70} with initial velocity $u_0$.
We use Theorem 2 of \cite{Wiedemann} with $\overline{v} = (v,0)$ and we define the trace-free matrix $\overline{u}$ by
$$\overline{u} = \overline{v} \otimes \overline{v} - \frac{|v|^2}{3}\mathbb{I}d. $$

Since $v$ is a solution of the Euler system, it follows that there exists a smooth, periodic (in space) pressure $\overline{q}$ such that the
triplet $(\overline{v},\overline{u},\overline{q})$ satisfies the linear system (1) from Theorem 2 in \cite{Wiedemann}. In addition, the conditions of Theorem 2 in \cite{Wiedemann}, that
$\overline{v} \in C^0([0,\infty);H_w(Q^3))$ and $\overline{u}(t,x)$ be a trace-free symmetric $3\times 3$ matrix, are also satisfied. Note that
\[e(\overline{v}(t,x),\overline{u}(t,x)) =  \frac{3}{2}\lambda_{max} (\overline{v}\otimes\overline{v} - \overline{u}) = \frac{|v|^2}{2},\]
where, for any symmetric matrix $M$, $\lambda_{max}(M)$ is the largest eigenvalue of $M$.
Next, take $g = g(t,x)$ to be a positive, $Q^3$-periodic and continuous function on $(0,\infty) \times \real^3$, belonging to $C_b^0([0,\infty);L^1(Q^3))$,
and define
\[\overline{e}(t,x) =  \frac{|v(t,x)|^2}{2} + g(t,x).\]
Then, using Theorem 2 in \cite{Wiedemann}, there exist infinitely many weak solutions $u \in C^0([0,\infty);H_w(Q^3))$
of the incompressible $3D$ Euler equations in $Q^3$ with initial data $(u_0,0)$, and such that for every $t \in (0,\infty)$ and almost
every $x \in Q^3$, $$\frac{|u(t,x)|^2}{2} =   \frac{|v(t,x)|^2}{2} + g(t,x).$$
We choose, for example
\[g(t,x) = \frac{t}{t^2+1} (1 + \sin^2(2\pi x_3)).\]
Clearly, $|u|^2/2$ is not constant with respect to any of the three spatial variables $x_1$, $x_2$ and $x_3$, which trivially implies that $u$  is not ei-$x_3$.
\end{proof}

\begin{rem}
Observe that the solution $u = u(t,x)$ satisfies $\|u(t,\cdot)\|_{L^2(Q^3)} > \|u(0,\cdot)\|_{L^2(Q^3)}$, for all $t>0$.
\end{rem}

This example is not the final word on this issue, since it is natural to restrict the search of weak solutions to a smaller
class, perhaps satisfying some physically motivated entropy-like criterion. Indeed, even in the viscous case, we ruled out spontaneous
symmetry breaking only for Leray-Hopf weak solutions, and not for weak solutions in general. In \cite{pllions}, P.-L. Lions introduced a
notion of generalized solution to the Euler equations which he called {\it dissipative solution}. He proved weak-strong uniqueness, in
this class, see Proposition 4.1 of \cite{pllions}, for flows satisfying certain regularity assumptions. The definition of dissipative solution, as given in \cite{pllions}, is complicated, but it was later noticed that weak solutions of the incompressible $3D$ Euler equations, which satisfy the weak energy inequality, are dissipative solutions in the sense of Lions, see \cite{DS10}, Proposition 1, for a proof of this fact. The weak-strong uniqueness of dissipative solutions, together with the observation regarding weak solutions which satisfy the weak energy inequality, imply that spontaneous symmetry breaking can be ruled out for these dissipative solutions. More precisely, we have the following result:

\begin{theorem} \label{lions}
Let $u_0 =(u_0^1,u_0^2) \in H(Q^2)$ be such that there exists a weak solution $\overline{u} \in C^0([0,\infty);H(Q^2))$ of the incompressible 2D Euler equations such that the symmetric part of $\nabla \overline{u}$ belongs to $L^1_{\loc}([0,\infty);L^{\infty}(Q^2))$.   Then any weak solution $u$ of the incompressible 3D Euler equations in $Q^3$,
with initial data $(u_0,0)$, which satisfies the weak energy inequality, i.e., such that, for all $t>0$,
\[\|u(t,\cdot)\|_{H(Q^3)} \leq \|(u_0,0)\|_{H(Q^3)} = \|u_0\|_{H(Q^2)},\]
is independent of $x_3$ (and is equal to $(\overline{u},0)$).
\end{theorem}

The proof of Theorem \ref{lions} is based on Proposition 4.1 of \cite{pllions} and on Proposition 1 of \cite{DS10}. We will not present a complete proof of Theorem \ref{lions} because this would exceed the scope of this work.  However, we will provide a brief outline of the proof in three steps.

\begin{enumerate}
\item The results contained in Proposition 4.1 of \cite{pllions} and in Proposition 1 of \cite{DS10} are stated and proved for flows in $\real^N$. The first step is to adapt these results to periodic flows in $Q^N$, which can be done in a straightforward manner.
\item Let $\overline{u}$ be the weak solution in the statement of Theorem \ref{lions}. Then $U\equiv (\overline{u},0)$ is a weak solution of the $3D$ Euler equations satisfying $U(0,\cdot)=(u_0,0)$, $U \in C^0([0,\infty);H(Q^3))$ and the symmetric part of $\nabla U$ belongs to  $L^1_{\loc}([0,\infty);L^{\infty}(Q^2))$. Hence, by (the adaptation of) Lions' Proposition 4.1, \cite{pllions}, any dissipative solution with the same initial velocity will be equal to $U$.
\item By (the adaptation of) Proposition 1, \cite{DS10}, any weak solution of the $3D$ Euler equations which satisfies the weak energy inequality will be a dissipative solution and, hence, equals $U$. Clearly, $U$ is independent of $x_3$.
\end{enumerate}

Note that the regularity requirement, which we wrote in terms of {\it existence} of a weak solution, is not very restrictive. Indeed, with initial vorticity in $L^{\infty}$, we already have existence and uniqueness
of a global weak solution in $C^0([0,\infty);H(Q^2))$ such that all first derivatives of velocity are in $L^{\infty}_{\loc}([0,\infty), BMO(Q^2))$, see \cite{Yudovich} and Theorem 7.1 in \cite{Torchinsky}. The condition in Lions' result is slightly more restrictive, and is certainly satisfied by strong solutions as in \cite{EM70}.

As we noted, any solution constructed using the strategy in Theorem \ref{DSz} will not satisfy the weak energy inequality, which places them out of the scope of Theorem
\ref{lions}. In \cite{DS10}, C. De Lellis and L. Sz\'ekelyhidi constructed examples of nonuniqueness of dissipative solutions of the Euler equations with $L^2$ initial velocities (this does not contradict the Yudovich criteria since vorticity of such initial data does not belong to $L^\infty$)  so uniqueness for dissipative solutions in general cannot hold. Up to now we have  no example of spontaneous symmetry breaking for dissipative solutions, but it would not be a surprise if the convex
integration techniques would allow the construction of such an example as well.

Ultimately, the most precise entropy criterion for weak solutions of the Euler equations is to be attained as a vanishing viscosity limit;  we can call
such solutions {\it viscosity solutions} of the Euler equations. (Observe that in the absence of physical boundaries, as in the present situation, any viscosity solution is a dissipative solution, see Proposition 4.2 of \cite{pllions}).

Our work already provides one result on retaining symmetry. If $u$ is a weak solution of the 3D Euler equations which is $Q^3$-periodic, with initial data $u_0$, and if $u_0$ is independent of $x_3$ then, as we have proved in the previous section, any Leray-Hopf weak solution of the Navier-Stokes equations with
initial data precisely equal to $u_0$ will be independent of $x_3$. It is easy to see that essentially any limit of $x_3$-independent flows will be  $x_3$-independent as well.  Restricting the notion of viscosity solutions
to those which are limits of Leray-Hopf weak solutions is quite reasonable, as those are the physically meaningful weak solutions of the Navier-Stokes equations. However, insisting
that viscosity solutions be limits of vanishing viscosity limits with exactly the same data might be too demanding.  In this sense, let us define a {\it viscosity weak solution} of the
Euler system with initial data $u_0$ as a solution which is a weak-star limit in $L^{\infty}((0,\infty);L^2(Q^3))$, as $\nu \to 0+$, of Leray-Hopf weak solutions of the $\nu$-
Navier-Stokes system in $Q^3$ with initial data $u_0^{\nu}$, where $u_0^{\nu} \to (u_0,0)$ strongly in $L^2$ when $\nu \to 0+$.  We will state and prove a result on retaining symmetry for viscosity solutions; we begin with a measure theory lemma.

\begin{lemma} \label{meas}
Let $f \in L^2(Q^3)$. Then $f$ is ei-$x_3$ if and only if $\partial_{x_3} f = 0$ in the sense of distributions.
\end{lemma}

\begin{proof}
The fact that $x_3$-independence  implies $\partial_{x_3} f = 0$ is an application of Fubini's Theorem. Indeed, let $\varphi \in C^{\infty}_{per}(Q^3)$. Fix $b \in (0,1)$ such that $f(x',a)=f(x',b)$ for almost every $a \in (0,1)$ and almost every $x'\in Q^2$. Let $A\subset (0,1)$ be defined by
\[A=\{a\in (0,1) \;|\; f(x',a)=f(x',b)\};\]
notice that $|A|=1$.
Then we have
\[\int_{Q^3} \partial_{x_3}\varphi (x) f(x) \,dx = \int_{(0,1)}\int_{Q^2} \partial_{x_3}\varphi (x',x_3) f(x',x_3) \,dx'dx_3 \]
\[ = \int_{A}\int_{Q^2} \partial_{x_3}\varphi (x',a) f(x',a) \,dx'da  = \int_A\int_{Q^2} \partial_{x_3}\varphi (x',a) f(x',b) \,dx'da \] \[=\int_{Q^2}f(x',b)\int_{A} \partial_{x_3}\varphi (x',a)  \, da dx' = \int_{Q^2}f(x',b)\int_{(0,1)} \partial_{x_3}\varphi (x',a)  \, da dx' = 0.\]

Conversely, assume that  $\partial_{x_3} f = 0$ in the sense of distributions. We write the Fourier series of $f$ as:
\[ f = \sum_{k \in \mathbb{Z}^3} \hat{f}(k) e^{2\pi ik\cdot x}.\]
Since $f \in L^2(Q^3)$, it follows that the truncations
\[f_N = f_N(x) \equiv \sum_{k \in \mathbb{Z}^3, |k|\leq N} \hat{f}(k) e^{2\pi ik\cdot x}\]
converges in $L^2$, and therefore, admits a subsequence, which we do not relabel, converging pointwise almost everywhere to $f$.

It can be verified that, for each fixed $N \in \mathbb{N}$,  $f_N$ is a function of $x_1$ and $x_2$ alone. Indeed, if $k=(k_1,k_2,k_3) \in \mathbb{Z}^3$ and $k_3 \neq 0$ then
\[\hat{f}(k)= \int_{Q^3} f(x) e^{-2\pi i k \cdot x} \, dx = \frac{-1}{2\pi i k_3} \int_{Q^3} f(x) \partial_{x_3} e^{-2\pi i k \cdot x} \, dx.\]
Hence, since trigonometric polynomials belong to $C^{\infty}_{per}(Q^3)$, it follows that, if $k_3 \neq 0$, then $\hat{f}(k) = 0$.
Consequently, $f_N$ is independent of $x_3$.

To conclude we note that, by Fubini's theorem, we have that, for almost all $a,b \in (0,1)$,  $f_N(x',a) \to f(x',a)$ and
$f_N(x',b) \to f(x',b)$ pointwise almost everywhere. It follows that $f(x',a)=f(x',b)$ for almost every $a$, $b \in (0,1)$, as desired.

\end{proof}

\begin{theorem} Let $u_0 \in H(Q^3)$ be ei-$x_3$ and let $u \in L^{\infty}((0,\infty);H(Q^3))$ be a weak solution of the 3D Euler equations in $Q^3$ with
initial data $u_0$. For each $\nu>0$, assume that there exists  $u^{\nu}$, a Leray-Hopf weak solutions of the Navier-Stokes equations in $Q^3$
with viscosity $\nu$ and with initial data $u_0^{\nu} \in H(Q^3)$, such that $u^{\nu} \rightharpoonup u$ in the sense of distributions in $(0,\infty) \times Q^3$, and
that there exists $C \geq \ds{\frac{27}{64}\|u_0\|_{L^2(Q^3)}^4}$, such that:
\begin{equation} \label{expcond}
 \|u_0 - u_0^{\nu}\|_{L^2(Q^3)} =  o\left( e^{-C/\nu^4} \right).
\end{equation}
Then $u$ is ei-$x_3$ for almost all time.
\end{theorem}

\begin{proof}
First consider $v^{\nu} = v^{\nu}(x,t)$ to be the Leray-Hopf weak solutions of the Navier-Stokes equations with viscosity $\nu > 0$ in $Q^3$ with initial data $u_0$. By Theorem \ref{2andHalfD},
$v^{\nu}$ is $x_3$-independent  for almost all time. Using the fact that $v^{\nu} \in C^0([0,\infty);H_w(Q^3))$ and Lemma \ref{meas} we can assume that $v^{\nu}$ is  $x_3$-independent for all time.
We write \[u = (u-u^{\nu}) + (u^{\nu} - v^{\nu}) + v^{\nu}.\] Using Theorem \ref{2andHalfD}, we have the following estimate:
\[\|v^{\nu} - u^{\nu}\|_{L^2(Q^3)} \leq \|u_0 - u_0^{\nu}\|_{L^2(Q^3)} \exp\left\{ \frac{27}{64\nu^4} \|u_0\|_{L^2(Q^3)}^4\right\}. \]

From the hypothesis, it follows that $u^{\nu} - v^{\nu} = o(1)$, as $\nu \to 0+$ in  $L^2(Q^3)$.

Let $\varphi = \varphi(t,x) \in C^{\infty}_c(Q^3)$ and choose $\eta \in C^{\infty}_{per}([0,\infty))$. We have:
\[\langle \eta\partial_{x_3}\varphi, u \rangle = \langle \eta\partial_{x_3}\varphi, u - u^{\nu} \rangle + \int_0^{\infty}\int_{Q^3} \eta\partial_{x_3}\varphi ( u^{\nu} - v^{\nu} )\,dxdt \]
\[+ \int_0^{\infty}\int_{Q^3} \eta\partial_{x_3}\varphi v^{\nu} \,dxdt.\]
Hence, using the fact that $v^{\nu}$ is $x_3$-independent we find, by Lemma \ref{meas},
\[\left|\langle \eta\partial_{x_3}\varphi, u \rangle \right| \leq  \left|\langle \eta\partial_{x_3}\varphi, u - u^{\nu} \rangle \right|+ \int_0^{\infty}|\eta|\|\partial_{x_3}\varphi \|_{L^2(Q^3)}\| u^{\nu} - v^{\nu}\|_{L^2(Q^3)}\,dt,
\]
which vanishes as $\nu \to 0^+$. This concludes the proof.

\end{proof}

\section{Comments and conclusions}

   One problem to be investigated is to try to extend the viscous stability result to flows in $\real^2 \times (0,L)$, periodic in the third variable. This could be attempted through the method developed in Section 2 or, perhaps, by adapting the work of Kozono and Taniuchi to this context. A more interesting, albeit difficult, class of problems is to consider perturbations which are not periodic, such as arise for compactly supported perturbations of Poiseuille flow in an infinite pipe. Another possible line of investigation is to  search for an example of inviscid symmetry breaking among dissipative solutions using convex integration techniques.

\section{Acknowledgements}
Bardos acknowledges the kind hospitality of the Weizmann Institute of Science, where part of this work was done. Lopes Filho's research is supported in part by CNPq grants 303089/2010-5 and  200434/2011-0. Nussenzveig Lopes' work is partially supported by CNPq grant 306331/2010-1 and CAPES grant 6649/10-6. Lopes Filho and Nussenzveig Lopes also acknowledge the support of the FAPESP Thematic Project 2007/51490-7, the CNPq Cooperation Project  490124/2009-7, and the hospitality of the Mathematics Department of the Univ. of California, Riverside. Niu's research was partially supported by National Youth grant, China (No. 11001184).
 Titi's work  was supported in part by the NSF grants DMS-1009950,
DMS-1109640 and DMS-1109645. Titi  also
acknowledges the support of the Alexander von Humboldt
Stiftung/Foundation and the Minerva Stiftung/Foundation.

\end{document}